\documentclass[10pt, reqno]{amsart}

\usepackage{amssymb,amsthm}
\usepackage{enumitem}
\usepackage{graphicx}
\usepackage{url}
\usepackage[margin=1in]{geometry}
\usepackage{hyperref}
\allowdisplaybreaks

\usepackage[T1]{fontenc}
\usepackage{lmodern}
\usepackage{microtype}

\graphicspath{{././figure/}}
\usepackage{thm-restate}

\newtheorem{theorem}{Theorem}[section]
\newtheorem{lemma}[theorem]{Lemma}
\newtheorem{proposition}[theorem]{Proposition}

\theoremstyle{definition}
\newtheorem{definition}[theorem]{Definition}

\theoremstyle{remark}
\newtheorem{remark}[theorem]{Remark}
\numberwithin{equation}{section}

\newcommand{\SU}{\mathrm{SU}}
\newcommand{\SL}{\mathrm{SL}}
\newcommand{\PSL}{\mathrm{PSL}}
\newcommand{\SO}{\mathrm{SO}}

\newcommand{\HH}{\mathbb{H}}
\newcommand{\CC}{\mathbb{C}}
\newcommand{\ZZ}{\mathbb{Z}}
\newcommand{\QQ}{\mathbb{Q}}
\newcommand{\RR}{\mathbb{R}}
\newcommand{\tr}{\operatorname{tr}}
\newcommand{\Gal}{\mathrm{Gal}}

\title{The homology cobordism group is not generated by graph homology $3$-spheres}

\author{Ben Mares}
\address{
Tensorial \\
Ferney-Voltaire, France
}

\author{Yuta Nozaki}
\address{
Department of Mathematics, Faculty of Science, Hokkaido University \\
Sapporo 060-0810 \\
Japan\vspace{-0.6em}}
\address{
International Institute for Sustainability with Knotted Chiral Meta Matter (WPI-SKCM$^2$), Hiroshima University \\
1-3-1 Kagamiyama, Higashi-Hiroshima, Hiroshima 739-8531 \\
Japan}
\email{nozaki@math.sci.hokudai.ac.jp}

\author{Masaki Taniguchi}
\address{Department of Mathematics, Graduate School of Science, Kyoto University, Kitashirakawa Oiwake-cho, Sakyo-ku, Kyoto 606-8502, Japan}
\email{taniguchi.masaki.7m@kyoto-u.ac.jp}

\begin{document}
\begin{abstract}
We show the homology cobordism group of homology $3$-spheres is not generated by graph homology $3$-spheres. Our technique is a combination of filtered instanton Floer theory and Gromov's simplicial volume. 
\end{abstract}
\maketitle

\setcounter{tocdepth}{1}

\section{Introduction}
\label{sec:Introduction}

Among homology cobordism groups, the three-dimensional integral
homology cobordism group $\Theta^3_{\mathbb Z}$ exhibits exceptional
behavior. Let $\Theta^n_{\mathbb Z}$ denote the smooth integral homology
cobordism group in dimension $n$, and let $\Theta^n$ denote the
homotopy-sphere cobordism group. 
Gonz\'{a}lez-Acu\~{n}a proved that $\Theta^n_{\mathbb Z}\cong \Theta^n$ for every $n\neq 3$
\cite{Gon70,Sav24}. 
By the work of Kervaire--Milnor \cite{KM63}, $\Theta^n$ is finite; in particular, the groups for $n=1,2, 4$ are trivial, while for $n\geq 5$ these groups are identified with the familiar finite
groups of oriented exotic spheres.
The contrast is even sharper in the PL and topological categories.
Kervaire proved that every PL homology $n$-sphere with $n\neq 3$ bounds
a contractible PL $(n+1)$-manifold \cite{Ker69}, and hence the
corresponding PL homology cobordism groups vanish outside dimension
three. In dimension
three, the smooth and PL categories agree, whereas the topological
homology cobordism groups vanish in all dimensions; in particular,
Freedman's work gives the dimension-three case \cite{Fre82}. Consequently, $\Theta^3_{\mathbb Z}$ is the exceptional
remaining case, and its structure is still a central open problem; see,
for example, \cite[Problem~3.69]{K3} and \cite{Sav24}.

Gauge theory and Floer theory have played central roles in the study of
$\Theta^3_{\mathbb Z}$. 
Furuta and Fintushel--Stern used instanton gauge theory to show that $\Theta^3_{\mathbb Z}$ contains a subgroup isomorphic to $\mathbb Z^\infty$ \cite{FS90,Fur90}.  
Fr{\o}yshov introduced a surjective homomorphism $h\colon \Theta^3_{\mathbb Z}\to\mathbb Z$ using $\SO(3)$-equivariant instanton Floer theory \cite{Fro02}, and Ozsv\'ath--Szab\'o obtained a similar homology cobordism invariant $d$ from Heegaard Floer theory~\cite{OS03}.
Manolescu's $\operatorname{Pin}(2)$-equivariant Seiberg--Witten Floer
theory produced the invariants $\alpha$, $\beta$, and $\gamma$; in
particular, $\beta$ is an integral lift of the Rokhlin invariant, leading to the
disproof of the triangulation conjecture~\cite{Man16} using a deep relation \cite{Mat78, GS80} between triangulation of topological manifolds and $\Theta^3_{\mathbb{Z}}$.  
Further applications of $\operatorname{Pin}(2)$-equivariant, involutive and filtered
Floer theories have revealed increasingly refined structure in $\Theta^3_{\mathbb Z}$ \cite{HM17, HMZ18,Dai20,  Sto20,HHL21, NST24}.  
In particular, Dai--Hom--Stoffregen--Truong \cite{DHST23} proved that $\Theta^3_{\mathbb Z}$ contains a direct summand isomorphic to $\mathbb Z^\infty$.  
See also \cite{FS85, FF00, Man03, Fro10,  Man14, Sto17, DM19, Fro23q2, DIST25,LS26,HSZ26, BS26}
for related developments in the study of homology cobordism.

Since many of the results above are obtained from Seifert homology $3$-spheres, largely because Floer-theoretic invariants are more computable in this setting, it is natural to ask what the geometric generators of $\Theta^3_{\mathbb{Z}} $ are.  
By Myers~\cite{Mye83}, every homology cobordism class contains a
hyperbolic representative, while Mukherjee~\cite{Muk20} showed that
$\Theta^3_\mathbb{Z}$ is generated by Stein fillable homology 3-spheres. 
An interaction between hyperbolic geometry and homology cobordism was also studied by Lin~\cite{FLin25}.

In the opposite direction, 
Hendricks--Hom--Stoffregen--Zemke~\cite{HHSZ} proved that the quotient of
$\Theta^3_{\mathbb Z}$ by the subgroup generated by Seifert fibered homology spheres is infinitely generated. 
In particular, $\Theta^3_{\mathbb Z}$ is not generated by Seifert fibered homology spheres. 
It was therefore natural to ask whether graph homology spheres, which form the
next class in the geometric hierarchy, generate $\Theta^3_{\mathbb Z}$. 
This question was explicitly raised in \cite{HHSZ}, \cite[Question~1.16]{NST24} and \cite[Problem~L]{Sav24}.  
Our main theorem gives a negative answer.

\begin{theorem}\label{main_theorem}
    The homology cobordism group of homology $3$-spheres is not generated by graph homology $3$-spheres.
\end{theorem}

Theorem~\ref{main_theorem} is a corollary of the following theorem, which relates our study to Gromov's simplicial volumes~\cite{Gro82}.  
We consider the following group seminorm: 
\[
\| \text{--} \| \colon \Theta^3_\mathbb{Z} \to \mathbb{R}_{\geq 0}
\]
defined as $\| [Y] \| = \inf_{Z\in [Y]} \|Z\|$, 
where $\|Z\|$ denotes the simplicial volume.
The subadditivity of the infimum follows from additivity of connected sum,
and graph manifolds have vanishing simplicial volume
\cite[Corollary~7.8]{Fri17}.
Moreover, by geometrization,
the null subgroup of this seminorm is exactly the subgroup $\Theta^3_G$ generated
by graph homology $3$-spheres\footnote{Indeed, by the prime decomposition theorem and geometrization, the simplicial volume of a closed oriented $3$-manifold is positive precisely when one of its prime summands has a hyperbolic JSJ piece. 
Moreover, finite-volume hyperbolic $3$-manifolds have a universal
positive lower bound on their volumes.  
Thus an integral homology $3$-sphere has zero simplicial volume if and only if it is a connected sum of graph homology $3$-spheres; see \cite{Gro82,Som81,Fri17}.}. 
Thus, our seminorm descends to a norm on the quotient
$\Theta^3_{\mathbb{Z}}/\Theta^3_G$, and our result can be expressed as follows.

\begin{theorem}\label{main:simplicial_volume}
  Let $Y$ be the homology $3$-sphere $S^3_{1/2}(5_2^*)$ obtained by Dehn surgery on the mirror image of the knot $5_2$ along slope $1/2$, and let $[Y]$ denote the corresponding class in $\Theta^3_{\mathbb{Z}}$. 
  Then $\| [Y] \| = \| Y \| \approx 2.19$. 
\end{theorem}

Theorem~\ref{main:simplicial_volume} implies that $Y$ achieves a positive infimum value for its simplicial volume, and therefore $[Y]$ exhibits a non-trivial element of the quotient group $\Theta^3_{\mathbb{Z}}/\Theta^3_G$.
This shows Theorem~\ref{main_theorem}.

Theorem~\ref{main:simplicial_volume} follows from Theorem~\ref{thm:invariance-volume} below.  
In order to state Theorem~\ref{thm:invariance-volume}, we need two tools: \emph{filtered instanton invariants} and \emph{Galois conjugations} of representations. 

The first tool is the filtered instanton invariant 
\[
r_0(Y) \in (0, \infty]
\]
of oriented homology $3$-spheres introduced by Nozaki--Sato--Taniguchi \cite{NST24}, which is invariant under smooth homology cobordism. 
Roughly speaking, $r_0(Y)$ measures a distinguished Chern--Simons level detected by the cohomology class called $[\theta]$ in the filtered instanton Floer complex. 
This level is represented by an irreducible flat $\SU(2)$-connection. 
In this paper, when we say that an irreducible representation $\rho\colon \pi_1(Y)\to \SU(2)$
\emph{uniquely realizes the $r_0$-invariant}, we mean that $\rho$ represents this
distinguished filtered level, uniquely up to $\SU(2)$-conjugacy.
The precise definition of $r_0$ and of what it means for a representation to realize $r_0$ will be given in Section~\ref{section: filtered instanton Floer}.

The second notion is Galois conjugation of an algebraic representation.  
For the role of Galois conjugation in the arithmetic study of Kleinian groups and hyperbolic $3$-manifolds, see, for example, \cite{MR03}.

A number field is a finite extension $F/\mathbb{Q}$.
A representation $\rho \colon \pi_1(Y) \to \SL(2,\mathbb{C})$ is said to be \emph{algebraic} if, up to equivalence, it factors through $\SL(2,F)\hookrightarrow \SL(2,\mathbb{C})$ for some number field $F$ and some field embedding $\iota \colon F \hookrightarrow \mathbb C$.  
 
For an algebraic representation $\rho_F$ and  each field embedding $
\sigma \colon F \hookrightarrow \mathbb{C}$,
we define the representation $
\rho_F^\sigma \colon \pi_1(Y) \to \SL(2,\mathbb{C})$
by applying $\sigma$ to every matrix entry:
\[
\rho_F^\sigma(g)
:=
\sigma\bigl(\rho_F(g)\bigr) \in \mathrm{SL}(2, \mathbb{C}).
\]
We refer to $\rho_F^\sigma$ as a \emph{Galois conjugate representation} of $\rho_F$.

The following is the main technical theorem in this paper.

\begin{restatable}{theorem}{volumeinvariance}
\label{thm:invariance-volume}
Suppose $Y$ is a hyperbolic integral homology $3$-sphere, and that an irreducible representation
$\rho \colon \pi_1(Y)\to \SU(2)$
uniquely realizes the $r_0$-invariant. 
Suppose that $\rho$ is algebraic and Galois conjugate to an $\mathrm{SL}(2, \mathbb{C})$-lift of the $\mathrm{PSL}(2, \mathbb{C})$ holonomy representation of the oriented hyperbolic structure of $Y$. 
If $Z$ is an integral homology $3$-sphere which is smoothly homology cobordant to $Y$, then
\[
\| Z \| \geq  \|Y\|= \frac{\operatorname{Vol}_{\mathrm{hyp}}(Y)}{v_3}>0, 
\]
where $\operatorname{Vol}_{\mathrm{hyp}}(Y)$ is the hyperbolic volume of $Y$ and $v_3\approx1.01494161$ denotes the volume of the regular ideal hyperbolic tetrahedron.
\end{restatable}

A sketch of the proof of Theorem \ref{thm:invariance-volume} is given as follows: 
The instanton invariant $r_0(Y)$ proves the existence of an $\SU(2)$-valued representation that extends over any integral homology cobordism whose restriction to $Y$ is $\rho$. Then we replace this representation twice. The weak Nullstellensatz allows us to replace the $\SU(2)$ extension with a fully algebraic $\SL(2,\CC)$ extension. Then Galois conjugation allows us to replace this algebraic representation with a representation that restricts to the holonomy representation on $Y$, so that it carries positive representation volume. 
See Definition~\ref{def:volume} for the definition of volumes of representations. 
By Stokes' theorem, representation volume is preserved across the homology cobordism, hence $Z$ admits an $\mathrm{SL}(2, \mathbb{C})$-representation with positive representation volume. Using Reznikov's volume inequality~\cite[Section~5.1]{Rez96} comparing representation volume and simplicial volume, we get the desired inequality. 

\begin{remark}
    More generally, the same argument as in Theorem~\ref{thm:invariance-volume} shows that if any Galois conjugate of the distinguished representation has non-zero volume, then
\[
\|[Y]\|
\geq
\frac{\left|\operatorname{Vol}(Y,\rho^\sigma)\right|}{v_3}
>0,
\]
where $\operatorname{Vol}(Y,\rho^\sigma)$ denotes the volume of the representation $\rho^\sigma$, even if $\rho^\sigma$ is not equivalent to the lift of the holonomy representation of $Y$. 
\end{remark}

\begin{remark}
   Similar statements hold for real-valued homology cobordism invariants derived from instanton theory with Chern--Simons filtrations such as $r_s(Y)$ \cite{NST24} for a general $s \in [-\infty, 0]$, $\Gamma_Y(k)$ \cite{Dai20} for $k\in \mathbb{Z}$, and analogous invariants introduced in \cite{DIST25}. 
\end{remark}

We shall focus on the hyperbolic homology sphere $S^3_{1/2}(5_2^*)$ and show the following theorem.

\begin{restatable}{theorem}{maincomputation} \label{thm:main-computation}
Let $Y=S^3_{1/2}(5_2^*)$ and let $\rho_1\colon \pi_1(Y)\to \SU(2)$ be an irreducible $\SU(2)$-representation uniquely realizing the $r_0$-invariant. 
Then, up to equivalence, $\rho_1$ is algebraic over a number field $F$, and there exists a field embedding $\sigma\colon F \hookrightarrow \mathbb{C}$ such that the corresponding Galois conjugate representation $\rho_1^\sigma\colon \pi_1(Y)\to \SL(2,\mathbb{C})$ is equivalent to an $\SL(2,\mathbb{C})$-lift of the holonomy representation coming from the hyperbolic structure of $Y$. 
\end{restatable}

The above two theorems immediately imply Theorem~\ref{main:simplicial_volume}, and hence
Theorem~\ref{main_theorem}. 

\begin{proof}[Proof of Theorem~\ref{main:simplicial_volume}]
Let $Y=S^3_{1/2}(5_2^*)$. 
By Theorem~\ref{thm:main-computation}, an irreducible $\SU(2)$-representation uniquely realizing $r_0(Y)$ is an algebraic representation admitting a Galois conjugate equivalent to an $\mathrm{SL}(2, \mathbb{C})$-lift of the hyperbolic
holonomy representation of $Y$. 
Hence, Theorem~\ref{thm:invariance-volume} implies that
\[
\inf_{Z\in [Y]} \|Z\|= \|Y\| = \frac{\operatorname{Vol}_{\mathrm{hyp}}(Y)}{v_3}. 
\]
A SnapPy computation~\cite{SnapPy} gives
$\operatorname{Vol}_{\mathrm{hyp}}(Y)\approx 2.22671790$.
By the proportionality principle (\cite[Section~2.2]{Gro82}, \cite[Theorem~6.2]{Thu80}), we have $\operatorname{Vol}_{\mathrm{hyp}}(Y)=v_3\|Y\|$, and thus $\|Y\|\approx2.19394$.
\end{proof}

\begin{remark}
The manifold $S^3_{1/2}(5_2^*)$ already appeared in \cite{NST24}
in connection with the question of whether $\Theta^3_{\mathbb Z}$ is
generated by graph homology $3$-spheres. 
Indeed, it was observed there that if the distinguished Chern--Simons critical value
$r_0(S^3_{1/2}(5_2^*))$ were irrational, then
$[S^3_{1/2}(5_2^*)]$ would not lie in the subgroup generated by graph
homology $3$-spheres (\cite[Proposition~1.17]{NST24}).
Based on a high-precision numerical computation, this value was conjectured to be irrational (\cite[Conjecture~1.19]{NST24}).  
This irrationality remains open; more generally, it is not known whether the Chern--Simons invariant of a flat $\SU(2)$-connection on a closed $3$-manifold can be irrational; see also \cite[Problem~3.64]{K3}. 
The present argument bypasses this unresolved irrationality problem. 
While Bloch groups are not essential for our proof, we briefly mention the powerful perspective they provide for rationality. 
Chern--Simons values of algebraic $\SU(2)$-representations are rational when the corresponding class in a Bloch group is torsion. 
Torsion is determined by the vanishing of the Borel regulator, which is essentially the vector of volumes of all Galois conjugate representations. 
Our observation that $\rho_1$ is Galois conjugate to a representation with non-zero volume thus proves that the Bloch group class is non-torsion. 
The converse of ``torsion implies rational'' is the subject of Ramakrishnan's conjecture \cite[Section~4]{neumann1995rationality}, and would thus imply the irrationality conjecture made in \cite{NST24}. 
For details, see \cite[Section~2, Theorem~5.1]{neumann1995rationality} and \cite{Neu04, GZ07}.
\end{remark}

Throughout this paper, all manifolds are assumed to be smooth, compact, orientable, and oriented unless otherwise stated.

\section*{AI disclosure}
The authors used generative AI tools (Claude Fable 5 and ChatGPT 6 Astra) for exploratory purposes, computation and proof cross-checks, and proofreading. All resulting mathematical arguments, computations, and references have been fully verified by the authors. No part of this manuscript was generated by AI; it was entirely human-written.

\section*{Acknowledgements}
The second and third authors were partially supported by JSPS KAKENHI Grant Numbers JP23K12974 and JP22K13921, respectively.
The authors thank Kouki Sato for the discussions.

\section{Preliminaries}

\subsection{Fundamental properties of volumes}
\label{subsec:volume}
We first review the notion of volume for representations.
For a closed oriented $3$-manifold equipped with a representation $\rho \colon \pi_1(Y) \to \PSL(2, \mathbb{C})$, we have the volume $\mathrm{Vol}(Y, \rho) \in \mathbb{R}$ as a natural generalization of hyperbolic volumes. 
See \cite{Gol82} and \cite{Dun99} for the general definition and history. 

Let $\widetilde{Y}$ be the universal cover of $Y$. 
Since $\HH^3$ is contractible, we can take a $\pi_1(Y)$-equivariant smooth map $f\colon \widetilde{Y} \to \mathbb{H}^3$, where $\pi_1(Y)$ acts on $\mathbb{H}^3$ via $\rho$ as orientation-preserving isometries. 
Let $\omega_\rho$ be the $3$-form obtained as the quotient of $f^* \mathrm{vol}_{\mathbb{H}^3}$ by $\pi_1(Y)$. 

\begin{definition}\label{def:volume}
The \emph{volume} $\mathrm{Vol}(Y, \rho)$ of a representation is defined by
\[
\mathrm{Vol}(Y, \rho) = \int_Y \omega_\rho \in \mathbb{R}. 
\]
\end{definition}

The resulting value is independent of the choice of $f$. 
In this section, we summarize the properties of volumes of representations. 

\begin{remark}
When we define volumes of representations, we use $\mathrm{PSL}(2, \mathbb{C})$ as the target Lie group. 
However, it is sometimes convenient to use $\mathrm{SL}(2, \mathbb{C})$. 
Note that for all homology $3$-spheres, any $\mathrm{PSL}(2, \mathbb{C})$-representation of the fundamental group uniquely lifts to an $\mathrm{SL}(2, \mathbb{C})$-representation. 
In particular, any hyperbolic homology $3$-sphere admits an $\SL(2,\CC)$-lift of its holonomy representation. 
For an $\mathrm{SL}(2, \mathbb{C})$-representation of a closed oriented $3$-manifold $Y$, its volume is defined by the volume of the corresponding $\mathrm{PSL}(2, \mathbb{C})$-representation via the quotient map $\mathrm{SL}(2, \mathbb{C}) \to \mathrm{PSL}(2, \mathbb{C})$. 
\end{remark}

Let $\Gamma$ be a group. 
For a representation $\rho\colon \Gamma\to\SL(2,\mathbb C)$ and an element $g\in\SL(2,\mathbb C)$, we write $g\rho g^{-1}$ for the representation defined by 
\[ 
(g\rho g^{-1})(\gamma) = g\rho(\gamma)g^{-1}. 
\] 
Two representations related in this way are said to be \emph{equivalent}, and the equivalence class of $\rho$ is denoted by $[\rho]$. 
Let $\kappa\colon\SL(2,\mathbb C)\to\SL(2,\mathbb C)$
be the involutive automorphism induced by complex conjugation:
\[
\kappa(A):=\overline{A},
\qquad
A\in\mathrm{SL}(2,\mathbb C).
\]

\begin{definition}
For a representation $\rho\colon \Gamma\to \SL(2,\mathbb C)$,
its \emph{complex conjugate representation} is defined by 
\[
\bar{\rho}:=\kappa\circ\rho \colon \Gamma \to \SL(2,\mathbb C).
\]
The character $\chi_\rho$ of $\rho$ is said to be \emph{real} if $\chi_\rho(\gamma) = \tr\rho(\gamma)$ is real for every $\gamma\in \Gamma$.
\end{definition}

Let $\Gamma=\pi_1(Y)$ for a closed oriented $3$-manifold $Y$.

\begin{lemma}
\label{lem:real-character-zero-volume}
Let
$\rho\colon\pi_1(Y)\to\SL(2,\mathbb C)$
be an irreducible representation and $\overline{\rho}$ be its complex conjugate. 
Then, we have 
\[
\operatorname{Vol}(Y,\rho)= -\operatorname{Vol}(Y, \bar{\rho}). 
\]
In particular, for an irreducible representation $\rho\colon\pi_1(Y)\to\SL(2,\mathbb C)$ whose character is real, we have 
\[
\operatorname{Vol}(Y,\rho)=0.
\]
\end{lemma}
\begin{proof} Let $c\colon\mathbb H^3\to\mathbb H^3$ be the orientation-reversing isometry defined by the complex conjugation on the first component $\mathbb{C} \times \mathbb{R}_{>0} =\mathbb H^3$. 
It satisfies 
\[ 
c(A\cdot x)=\overline{A}\cdot c(x) 
\] 
for $A\in\SL(2,\mathbb C)$, where the action of $\SL(2,\mathbb C)$ on $\mathbb H^3$ is understood through its quotient $\PSL(2,\mathbb C)$. 
If $f\colon\widetilde Y\to\mathbb H^3$ is a $\rho$-equivariant map, then $c\circ f$ is $\overline{\rho}$-equivariant. 
Since $c$ reverses the orientation of $\mathbb H^3$, 
\[ 
c^*\mathrm{vol}_{\mathbb H^3} = -\mathrm{vol}_{\mathbb H^3}. 
\] 
This shows the first equality. 

Now suppose that $\rho$ has real character. 
Then $\chi_{\bar{\rho}} = \overline{\chi_\rho} = \chi_\rho$. 
Since $\rho$ is irreducible, $\bar{\rho}$ is equivalent to $\rho$ in $\SL(2,\mathbb C)$. 
Since equivalent representations have the same volume, 
\[ 
\operatorname{Vol}(Y,\rho) = \operatorname{Vol}(Y,\bar{\rho}) = -\operatorname{Vol}(Y,\rho), 
\] 
and hence $\operatorname{Vol}(Y,\rho)=0$. 
\end{proof}

A standard Stokes argument proves the following.

\begin{lemma}
\label{lem:volume-cobordism}
Let $W$ be a compact oriented $4$-dimensional cobordism with $\partial W=-Y\sqcup Z$,
and let $\tilde{\rho}\colon
\pi_1(W)\to\SL(2,\mathbb C)$
be a representation. Denote its restrictions to the boundary
components by $\rho_Y$ and $\rho_Z$. 
Then,
\[
\operatorname{Vol}(Y,\rho_Y)
=
\operatorname{Vol}(Z,\rho_Z).
\]
\end{lemma}

\begin{proof}
Let $\widetilde W$ be the universal cover of $W$. 
Since $\mathbb H^3$ is contractible, there exists a $\tilde{\rho}$-equivariant map $f\colon \widetilde W\to\mathbb H^3$. 
The $3$-form $f^*\mathrm{vol}_{\mathbb H^3}$ is $\pi_1(W)$-invariant and
therefore descends to a $3$-form $\omega_{\tilde{\rho}}$ on $W$.
Since the volume form $\mathrm{vol}_{\HH^3}$ is closed, so is
$\omega_{\tilde{\rho}}$. Hence, by Stokes' theorem,
\[
0=\int_W d\omega_{\tilde{\rho}}
=\int_{\partial W}\omega_{\tilde{\rho}}
=-\operatorname{Vol}(Y,\rho_Y)
+\operatorname{Vol}(Z,\rho_Z).
\]
This proves the assertion.
\end{proof}

\section{Results from filtered instanton Floer theory}

The purpose of this section is to prove the following theorem stated in Section~\ref{sec:Introduction}. \volumeinvariance* 

We first recall the following standard consequence of Hilbert's weak Nullstellensatz (see, for example, \cite[Theorem~4.19]{Eis95}).

\begin{lemma}\label{lem:weak-nullstellensatz} Let $k$ be an algebraically closed field and let $f_1,\ldots,f_r\in k[x_1,\ldots,x_N]$. If the system 
\[ f_1(x)=\cdots=f_r(x)=0 \] has a solution in $K^N$ for some field extension $K/k$, then it has a solution in $k^N$.
\end{lemma} 

\begin{proposition} \label{prop:real-algebraic-transfer} 
Let $W$ be a compact connected oriented cobordism from a closed oriented $3$-manifold $Y$ to another $Z$, and suppose that there exists a representation $\tilde{\rho}\colon \pi_1(W)\to \SL(2,\mathbb C)$. 
Let $\rho_Y = \tilde{\rho}|_{\pi_1(Y)}$. Suppose that $\rho_Y$ is algebraic, that is, there exist a number field $F\subset\mathbb C$ and an element $g\in\SL(2,\mathbb C)$ such that $\rho_F = g\rho_Yg^{-1}$ has its image in $\SL(2,F)$. 
Then there exist a finite extension $L/F$ and a representation 
\[ 
\tilde{\rho}_L\colon \pi_1(W)\to\SL(2,L) 
\] 
satisfying $ \tilde{\rho}_L|_{\pi_1(Y)}=\rho_F$.
\end{proposition}

\begin{proof}
By assumption,
$g\tilde{\rho}g^{-1}\colon \pi_1(W)\to\SL(2,\mathbb C)$
is a representation whose restriction to $\pi_1(Y)$ is $\rho_F$.
Choose a finite presentation
$\pi_1(W)=\langle x_1, \ldots, x_n \mid R_1, \ldots, R_m\rangle$.
Let $y_1, \ldots, y_k$ be generators of $\pi_1(Y)$, and write $w_j(x_1, \ldots, x_n)$ for their images under the inclusion-induced homomorphism
$\pi_1(Y) \to \pi_1(W)$.

For each $i$, we introduce a variable 
\[
 X_i=
\begin{pmatrix}
a_i&b_i\\
c_i&d_i
\end{pmatrix} \in M(2,\mathbb C).
\]
The condition that the $X_i$'s define an $\SL(2,\mathbb C)$-representation
of $\pi_1(W)$ whose restriction to $\pi_1(Y)$ is $\rho_F$ is written as 
\[
\det X_i=1,\ R_\ell(X_1, \ldots, X_n) = I_2,  \text{ and } 
w_j(X_1, \ldots, X_n) = \rho_F(y_j).
\]
This condition forms a finite system of polynomial equations for the entries of the matrix $X_i$ with coefficients in $F$, and hence in $\overline{\mathbb Q}$. 
Indeed, since $\det X_i=1$, the inverse $X_i^{-1}$ occurring in
the words $R_\ell$ and $w_j$ is written as a polynomial in the entries of
$X_i$.

The representation 
\[
g\tilde{\rho}(x_i)g^{-1}=X_i
\]
gives a solution of this system over $\mathbb C$. 
Since $\overline{\mathbb Q}$ is algebraically
closed, Lemma~\ref{lem:weak-nullstellensatz}, applied to the field
extension $\overline{\mathbb Q}\subset\mathbb C$, shows that the system has a solution over $\overline{\mathbb Q}$. 
Hence, there exists a representation
$\tilde{\rho}_{\mathrm{alg}}\colon \pi_1(W)\to\SL(2,\overline{\mathbb Q})$ satisfying $\tilde{\rho}_{\mathrm{alg}}|_{\pi_1(Y)}=\rho_F$.

Only finitely many algebraic numbers occur among the entries of the $n$ matrices $\tilde{\rho}_{\mathrm{alg}}(x_1), \ldots, \tilde{\rho}_{\mathrm{alg}}(x_n)$.
Let $L$ be the field obtained from $F$ by adjoining these algebraic numbers.
Then $L/\mathbb Q$ is a finite extension, and hence $L/F$ is finite.
Moreover, $\tilde{\rho}_{\mathrm{alg}}(\pi_1(W))\subset\SL(2,L)$.
Setting $\tilde{\rho}_L = \tilde{\rho}_{\mathrm{alg}}$ completes the proof.
\end{proof}

\subsection{Filtered instanton Floer theory} \label{section: filtered instanton Floer}

In this subsection, we recall a part of filtered instanton Floer theory which will be used in the proof of Theorem~\ref{thm:invariance-volume}.
We follow the convention of \cite{NST24}. 
We only use the invariant $r_0$, and all instanton Floer groups are taken with coefficients in $\mathbb{Q}$. 
For details, see \cite{NST24}. 

Let $Y$ be an oriented integral homology $3$-sphere. 
We denote by $P_Y=Y\times \SU(2)$ the product $\SU(2)$-bundle and by $\theta$ the product connection. 
Let $\mathcal A^*(Y)$ be the space of irreducible $\SU(2)$-connections on $P_Y$ with a certain Sobolev completion, and let
\[
\mathcal G_0(Y)
=
\{g\colon Y\to \SU(2)\mid \deg(g)=0\}
\]
again with a completion. 
We put $\widetilde{\mathcal B}(Y) = \mathcal A^*(Y)/\mathcal G_0(Y)$.
The Chern--Simons functional
\[
\operatorname{cs}_Y\colon
\widetilde{\mathcal B}(Y)\to \mathbb{R}
\]
is given by
\[
\operatorname{cs}_Y(a)
=
\frac{1}{8\pi^2}
\int_Y
\operatorname{Tr}
\left(
a\wedge da+\frac{2}{3}a\wedge a\wedge a
\right).
\]

Let $\widetilde{\mathcal R}^*(Y)
=
\mathcal A_{\mathrm{flat}}(Y)/\mathcal G_0(Y)$, where $\mathcal A_{\mathrm{flat}}(Y)$ is the space of irreducible flat $\SU(2)$-connections with completion. 
We write
\[
\Lambda_Y^{*}
=
\operatorname{cs}_Y
\left(
\widetilde{\mathcal R}^{*}(Y)
\right)
\subset \mathbb{R}
\]
for the set of irreducible Chern--Simons critical values. 
The group $\mathbb{Z}$ acts on the set of critical values by translation, and $\Lambda_Y^{*}/ \mathbb{Z} $ becomes a finite subset of $\mathbb{R}/\mathbb{Z} =S^1$. 
Note that $\Lambda_Y^{*}$ is a topological invariant of $Y$.

Also, we define 
\[
{\mathcal R}^{*}(Y) : = \widetilde{\mathcal R}^{*}(Y)/ \mathbb{Z} \cong \mathcal{A}_{\mathrm{flat}}(Y) / \{ g \colon Y \to \SU(2) \}, 
\]
where $\mathbb{Z} =\{ g \colon Y \to \SU(2) \}  / \mathcal G_0(Y) $. 
Then, ${\mathcal R}^{*}(Y)$ is identified with the irreducible part of the $\SU(2)$-representation variety
\[
\bigl(\mathrm{Hom} (\pi_1(Y), \SU(2)) \setminus \{ [\theta]\} \bigr)/ \SU(2) 
\]
of $Y$, where $\theta$ denotes the trivial representation. 

For a sufficiently small non-degenerate regular holonomy perturbation
$\pi$, let
\[
\operatorname{cs}_{Y,\pi}
\colon
\widetilde{\mathcal B}(Y)\to\mathbb{R}
\]
denote the perturbed Chern--Simons functional, and let $
\widetilde{\mathcal R}^{*}_{\pi}(Y)$
be its irreducible critical set. Each
$a\in\widetilde{\mathcal R}^{*}_{\pi}(Y)$ is equipped with its
integer-valued Floer index
$\operatorname{ind}(a)\in\mathbb Z$.

Following \cite{NST24}, for a positive regular value $r$, we consider
the Chern--Simons filtered instanton chain complex $CI^{[0,r]}_*(Y;\pi)$.
Here, the notation $[0,r]$ is understood with the convention of
\cite{NST24} near the reducible critical value $0$. 
In particular, the lower endpoint is chosen slightly below $0$ so that the trivial connection is not a generator of the irreducible Floer complex.

There is a cochain
\[
\theta_Y^{[0,r]}
\colon
CI^{[0,r]}_1(Y;\pi)
\to
\mathbb Q
\]
defined by
$
\theta_Y^{[0,r]}(a)
=
\#\left(
\mathcal M_Y(a,\theta)_\pi/\mathbb R
\right)$,
where $\mathcal M_Y(a,\theta)_\pi$ denotes the 1-dimensional part of the moduli space of perturbed ASD connections on $Y\times\mathbb R$
converging to $a$ and $\theta$ at the two ends with respect to a product Riemann metric.  This is a filtered version of Donaldson's $D_1$-map introduced in \cite{Don02}. 
The corresponding
cohomology class
$
[\theta_Y^{[0,r]}]
\in
I^1_{[0,r]}(Y;\mathbb Q)$
is independent of the auxiliary choices.

The invariant $r_0(Y)$ is defined by
\[
r_0(Y)
=
\sup
\left\{
r>0
\;\middle|\;
[\theta_Y^{[0,r]}]=0
\right\} \in (0, \infty ].
\]
More precisely, the supremum is taken over admissible regular values
$r$, as in \cite[Definition~3.2]{NST24}. One of the basic properties of
$r_0$ is that if
$
r_0(Y)<\infty$, then
$r_0(Y)\in\Lambda_Y^{*}$.
That is, a finite value of $r_0(Y)$ is realized as the Chern--Simons value
of an irreducible flat $\SU(2)$-connection.

\begin{definition}
Let $Y$ be an oriented integral homology $3$-sphere and let $
\rho\colon \pi_1(Y)\to \SU(2)$
be an irreducible representation. We say that $\rho$ \emph{uniquely
realizes} the $r_0$-invariant if
\[
\operatorname{cs}_Y(\tilde{\rho})=r_0(Y)
 \text{ and }
\left\{
[\tilde{\alpha}]\in\widetilde{\mathcal R}^*(Y)
\;\middle|\;
\operatorname{cs}_Y(\tilde{\alpha})=r_0(Y)
\right\}
=
\{[\tilde{\rho}]\}, 
\]
where $\tilde{\rho} \in \widetilde{\mathcal{R}}^*(Y)$ denotes a lift of $\rho$ with respect to the covering map 
\[
\widetilde{\mathcal{R}}^*(Y) \to {\mathcal{R}}^*(Y). 
\]
\end{definition}

\begin{proposition}
\label{prop:SU2-extension}
Let $W$ be an integral homology cobordism from an oriented integral
homology $3$-sphere $Y$ to another $Z$. 
Suppose that $\rho\colon \pi_1(Y)\to \SU(2)$
is an irreducible representation which uniquely realizes $r_0(Y)$.
Then there exists a representation $
\tilde{\rho}\colon \pi_1(W)\to \SU(2)$
such that $\tilde{\rho}|_{\pi_1(Y)}$ is equivalent to $\rho$.
\end{proposition}

\begin{proof}
Let $r=r_0(Y)$.
Since $W$ is an integral homology cobordism, the homology cobordism
invariance of $r_0$ gives $r_0(Z) = r < \infty $.
Then, by \cite[Theorem~3.7]{NST24} (see also \cite[Theorem~3.10]{NST24} and \cite[Section~6.1]{ADMT23}), we obtain irreducible representations $\rho_Y'\colon\pi_1(Y)\to \SU(2)$ and $\rho_Z\colon\pi_1(Z)\to \SU(2)$
which are restrictions of a common representation $\tilde{\rho}\colon\pi_1(W)\to \SU(2)$.
Moreover, as follows from the proof of
\cite[Theorem~3.7]{NST24}, these representations satisfy
\[
\operatorname{cs}_Y(\rho_Y')
=
\operatorname{cs}_Z(\rho_Z)
=
r.
\]

Since $\rho$ uniquely realizes $r_0(Y)$, we have $[\rho_Y']=[\rho]$.
Thus, after conjugating $\tilde{\rho}$ by an element of $\SU(2)$,
we may assume that $
\tilde{\rho}|_{\pi_1(Y)}=\rho$. 
This proves the proposition.
\end{proof}

\begin{remark}
    The argument showing extendability of $\SU(2)$-representations given in Proposition~\ref{prop:SU2-extension} has also been developed in \cite[Section~8]{Tan22} under more general assumptions on $\SU(2)$-representation varieties. 
    See also \cite{ADHLP22} for related arguments. 
\end{remark}

Now, we give the proof of Theorem~\ref{thm:invariance-volume}. 

\begin{proof}[Proof of Theorem~\ref{thm:invariance-volume}]
Let $W$ be an integral homology cobordism from $Y$ to $Z$. 
By Propositions~\ref{prop:SU2-extension} and~\ref{prop:real-algebraic-transfer}, after enlarging the number field if necessary, there is an algebraic representation $\tilde{\rho}_L\colon\pi_1(W)\to\SL(2,L)$ extending an algebraic model $\rho_F$ of $\rho$. 
Let $\sigma\colon F\hookrightarrow\CC$ be an embedding for which $\rho_F^\sigma$ is equivalent to an $\SL(2,\CC)$ lift of the holonomy representation of $Y$.
Since $L/F$ is an algebraic extension and $\CC$ is algebraically closed, one can extend $\sigma$ to an embedding $\tilde{\sigma}\colon L\hookrightarrow\CC$. 
The representation $\tilde{\rho}_L^{\tilde{\sigma}}$ then restricts to $\rho_F^\sigma$ on $Y$. Denote its restriction to $Z$ by $\rho_Z$. 
By Lemma~\ref{lem:volume-cobordism} and Reznikov's volume inequality \cite[Section~5.1]{Rez96},
\[
\mathrm{Vol}_{\mathrm{hyp}}(Y)
=|\mathrm{Vol}(Y,\rho_F^\sigma)|
=|\mathrm{Vol}(Z,\rho_Z)|
\leq v_3\|Z\|.
\]
Since $Y$ is hyperbolic, $\operatorname{Vol}_{\mathrm{hyp}}(Y)=v_3\|Y\|$. 
Hence,
\[
\|Z\|\geq\|Y\|=\frac{\operatorname{Vol}_{\mathrm{hyp}}(Y)}{v_3}>0.
\]
This completes the proof.
\end{proof}

\begin{remark}\label{remark_inj}
Under the assumption of Theorem~\ref{thm:invariance-volume}, 
the representation
$
\rho\colon \pi_1(Y)\to \SU(2)$
uniquely realizing $r_0(Y)$ is faithful. Indeed, let
$\rho_F\colon\pi_1(Y)\to\SL(2,F)$ be an algebraic model of $\rho$.
By assumption, there is a field embedding
$\sigma\colon F\hookrightarrow\mathbb C$
such that the projectivization of $\rho_F^\sigma$ is equivalent to the holonomy representation of $Y$, which is faithful.
Thus, if $\rho(\gamma) = I_2$ for some $\gamma \in\pi_1(Y)$, then
$\rho_F^\sigma(\gamma) = I_2$, and hence $\gamma=1$.

Now let $W$ be any integral homology cobordism from $Y$ to an
integral homology $3$-sphere $Z$. 
By Proposition~\ref{prop:SU2-extension}, $\rho$ extends, up to equivalence,
to a representation
$\tilde{\rho}\colon\pi_1(W)\to\SU(2)$. 
If $\gamma \in\pi_1(Y)$ lies in the kernel of the inclusion-induced homomorphism
$\pi_1(Y) \to \pi_1(W)$, then $\rho(\gamma) = \tilde{\rho}(1) = I_2$.
Since $\rho$ is faithful, we have $\gamma=1$. 
Therefore, $\pi_1(Y) \to \pi_1(W)$
is injective.
\end{remark}

\section{Computations for a hyperbolic homology 3-sphere}
The goal of this section is to establish the following explicit computation. 
\maincomputation*

This follows from the combination of the following two propositions and one lemma. 
For $Y=S^3_{1/2}(5_2^*)$, we canonically identify $\pi_1(Y)$ with $\pi_1(-Y)$.

\begin{figure}[htbp]
    \centering
    \includegraphics[width=0.4\linewidth]{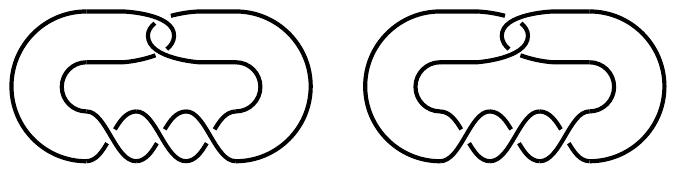}
    \caption{The knot $5_2$ and its mirror image $5_2^\ast$.}
    \label{fig:5_2}
\end{figure}

\begin{proposition}
\label{prop:main-computation-r0}
There exists an irreducible representation $\rho_1\colon \pi_1(Y)\to \SU(2)$, unique up to equivalence, such that $\operatorname{cs}_Y(\rho_1)=r_0(Y)$.
In particular, $\rho_1$ uniquely realizes the $r_0$-invariant.
\end{proposition}

\begin{proof}
This follows from the computation in \cite[Theorem~1.18 and Section~7]{NST24}.
There, the irreducible $\SU(2)$-representations of $\pi_1(Y)$ are described using the Riley polynomial of $5_2$, and their Chern--Simons values are computed using the Kirk--Klassen formula~\cite{KirkKlassen90}. 
The computation shows that
\[
r_0(Y) \approx 0.00176489047864885113
\]
and that exactly one equivalence class of irreducible $\SU(2)$-representations has this Chern--Simons value. 
Hence, this representation uniquely realizes $r_0(Y)$.
\end{proof}

We recall from \cite[Section~7]{NST24} the Riley polynomial and the $A$-polynomial.
Let $K$ be a $2$-bridge knot and let $E(K)$ denote the exterior of $K$.
Then the knot group $\pi_1(E(K))$ admits a presentation of the form $\langle x, y \mid wx = yw \rangle$, where $w$ is a certain word in $x$ and $y$.
For $M \in \CC\setminus\{0\}$ and $u \in \CC$, let $\rho_{M,u}$ denote the representation of the free group $\langle x, y \mid \text{--} \rangle$ of rank $2$ given by 
\[
\rho_{M,u}(x) = 
\begin{pmatrix}
 M & M^{-1} \\
 0 & M^{-1}
\end{pmatrix},
\quad
\rho_{M,u}(y) = 
\begin{pmatrix}
 M & 0 \\
 -Mu & M^{-1}
\end{pmatrix}.
\]
The \emph{Riley polynomial} of $K$ (for the above presentation) is defined by $\phi(M,u) = w_{11}+(1-M^2)w_{12} \in \ZZ[M^{\pm 1},u]$, where $w_{ij}$ is the $(i,j)$-entry of $\rho_{M,u}(w)$ (see \cite{Ril84}).
Note that the notation is changed from \cite{NST24} by $\varepsilon \sqrt{t} = M$. 
Then $\rho_{M,u}$ induces a representation of $\pi_1(E(K))$ if and only if $\phi(M,u)=0$.
Moreover, any irreducible representation of $\pi_1(E(K))$ is equivalent to $\rho_{M,u}$ for some $M$ and $u$.
For $K = 5_2$, we have $w = [y, x^{-1}]^2$ and 
\[
\phi(M,u)
=
-(M^4+M^{-4})u
+
(M^{-2}+M^2)(2+3u+2u^2)
-
(3+6u+3u^2+u^3).
\]

Let $\lambda$ and $\mu$ be the preferred oriented longitude and meridian of a knot $K$. 
For a representation $\rho \colon \pi_1(E(K)) \to \SL(2,\CC)$, up to equivalence, we may assume $\rho(\lambda)$ and $\rho(\mu)$ are upper triangular matrices since they commute.
The $(1,1)$-entries of these matrices are related via the \emph{$A$-polynomial} $A_K(L, M) \in \ZZ[L, M]$ of $K$ (see \cite{CCGLS94}).
At least for $2$-bridge knots, the vanishing of the $A$-polynomial imposes a necessary condition on eigenvalues $L$ and $M$ to arise from an irreducible representation of $\pi_1(E(K))$.
For $K= 5_2$, we have
\[
A_{5_2}(L, M) = -L^3 -M^{14} +L^2 (1 -2M^2 -2M^4 +M^8 -M^{10}) +LM^4 (-1 +M^2 -2M^6 -2M^8 +M^{10}).
\]

\begin{proposition}
\label{prop:main-computation-algebraic}
For a representation $\rho = \rho_{M,u} \colon \pi_1(Y) \to \SL(2,\CC)$, there exists a number field $F\subset\mathbb C$ such that the image of $\rho$ lies in $\SL(2,F)$.
\end{proposition}

\begin{proof}
Since $\rho$ factors through the $-1/2$-surgery, we have $\rho(\mu) \rho(\lambda)^{-2} = I_2$, and hence $M = L^2$.
Since $A_{5_2}(L, L^2)$ is a non-zero polynomial with integer coefficients, $L$ is algebraic over $\mathbb Q$, and so is $M$. 

The equation $\phi(M,u) = 0$ is a non-zero polynomial equation in $u$ with coefficients in $\mathbb Q(M)$. 
It follows that $u$ is an algebraic number.
Let $F = \mathbb Q(M, u)$.
Then, $F$ is a number field and the matrices $\rho(x)$ and $\rho(y)$ lie in $\SL(2,F)$. 
Thus, $\rho$ restricts to a representation $\pi_1(Y)\to \SL(2,F)$.
\end{proof}

\begin{table}[htbp]
\centering
$
\begin{array}{r|r|r|r|r}
 L+L^{-1} & M & u & \text{type} & |\mathrm{Vol}| \\\hline\hline
 1.61803 & 0.309017 +0.951057\, i & -1.00000 & \rho_5 & 0 \\
 -0.618034 & -0.809017 +0.587785\, i & -1.00000 & \rho_2  & 0  \\\hline
 -1.77462 & 0.574643 +0.818404\, i & -2.41421 & \rho_3 &  0  \\
 1.36041 & -0.0746432 +0.99721\, i & -2.41421 & \rho_8 &  0 \\
 0.185885 & -0.982723 +0.185081\, i & 0.414214 & \text{--} & 0 \\
 2.22833 & 2.57747 & 0.414214 & \text{--} & 0 \\\hline
 1.96292 & 0.926534 +0.37621\, i & -0.0755806 & \rho_1 & 0  \\
 -1.56046 & 0.217519 +0.976056\, i & -2.16991 & \rho_6 & 0  \\
 1.25785 & -0.208911 +0.977935\, i & -3.62043 & \rho_7 &  0  \\
 -1.15511 & -0.332861 +0.942976\, i & -1.6911 & \rho_4 & 0 \\
 0.2474 -0.404336\, i & -0.647653 +0.160567\, i & 0.278508 -0.958302\, i & \text{holonomy} & 2.2267179039 \\
 0.2474 +0.404336\, i & -1.45463 +0.360633\, i & 0.278508 +0.958302\, i & \text{holonomy} & 2.2267179039
\end{array}
$
\vspace*{5pt}
\caption{The $12$ irreducible $\SL(2,\CC)$-representations of $\pi_1(S^3_{-1/2}(5_2))$, where $\rho_i$'s are the $\SU(2)$-representations in \cite[Table~2]{NST24}.
These are obtained from $p_2$, $p_4$, and $p_6$ in this order.}
\label{tab:SL2reps}
\end{table}

Before proving Theorem~\ref{thm:main-computation}, we give a table of irreducible $\SL(2,\CC)$-representations of $\pi_1(-Y)$.
To obtain Table~\ref{tab:SL2reps}, we begin with
\[
A_{5_2}(L, L^2) = L^{14} (1+L)^3 p_2(L+L^{-1}) p_4(L+L^{-1}) p_6(L+L^{-1}),
\]
where
\[
p_2(x)= x^2-x-1, \quad p_4(x)= x^4-2 x^3-3 x^2+6 x-1, \quad p_6(x)= x^6-x^5-4 x^4+3 x^3+3 x^2-2 x+1.
\]
By Remark~\ref{rem:L+1} below, we ignore the factors $L$ and $1+L$.
Solving $p_i(x)=0$ numerically, we obtain $12$ candidates of $\tr \rho(\lambda) = L+L^{-1}$.
For each pair $(L_0, M_0)$ obtained from the candidates, we know $L_0, M_0 \notin \{-1,0,1\}$.
The $A$-polynomial of $5_2$ is obtained as the resultant of $\phi(M, u)$ and $L-\rho(\lambda)_{11}$ with respect to $u$, up to multiplication by a power of $M$.
The resultant description and that the Riley polynomial is monic in $u$ up to a sign ensure, by \cite[Theorem~3.1]{CoLo96}, the existence of a value $u_0$ such that the representation $\rho_{M_0,u_0}\colon \pi_1(E(K)) \to \SL(2,\CC)$ 
satisfies $(L_0, M_0) = (\rho_{M_0,u_0}(\lambda)_{11}, \rho_{M_0,u_0}(\mu)_{11})$.
(See \cite{KiNo22} for the argument based on the resultants,
or, alternatively, we can show the existence using the proof of \cite[Proposition~4.4]{CoLo96}.)
Since $L_0$ is not $\pm 1$, then the eigenvalues $L_0$ and $L_0^{-1}$ of $\rho_{M_0,u_0}(\lambda)$ are distinct, so $\rho_{M_0,u_0}(\lambda)$ and $\rho_{M_0,u_0}(\mu)$ are simultaneously diagonalizable, and the scalar filling relation $M_0 L_0^{-2} =1$ implies the matrix filling relation $\rho_{M_0,u_0}(\mu) \rho_{M_0,u_0}(\lambda)^{-2} =I_2$.
Note that the representations corresponding to $(L_0, M_0)$ are equivalent to those of $(L_0^{-1}, M_0^{-1})$.
Therefore, there are at least $12$ representations of $\pi_1(-Y)$.
On the other hand, among the three solutions of the equation $\phi(L_0^2, u)=0$ for $u$, we can numerically check that two of them do not satisfy $\rho(\mu)\rho(\lambda)^{-2} = I_2$.
Hence, there exist exactly $12$ representations as listed in Table~\ref{tab:SL2reps}.
Note that the reason why most volumes are zero will be explained in the proof of Theorem~\ref{thm:main-computation}.

\begin{remark}
\label{rem:L+1}
The case $L=0$ is incompatible with $\rho(\lambda) \in \SL(2,\CC)$.
When $L=-1$, the equations $A_{5_2}(L,M)=0$ and $L^2 = M$ have a unique solution $(L,M) = (-1,1)$.
Then $u$ must satisfy $1-2u+u^2-u^3 = 0$.
However, none of the three solutions satisfies $\rho(\mu) \rho(\lambda)^{-2} = I_2$.
\end{remark}

Finally, we give the proof of Theorem~\ref{thm:main-computation}. 

\begin{proof}[Proof of Theorem~\ref{thm:main-computation}]
First, recall that the representation $\rho_1$ in Proposition~\ref{prop:main-computation-r0} is algebraic over a number field $F$ by Proposition~\ref{prop:main-computation-algebraic}.
The roots of $p_2$ and $p_4$ are real, and $p_6$ has four real roots and a pair of complex roots.

Let $\rho$ be any representation corresponding to a real root.
Then, $\tr \rho(\lambda) = \tr \bar{\rho}(\lambda)$, where $\bar{\rho}$ is the complex conjugate of $\rho$.
Since the $12$ representations are characterized by the trace on $\lambda$, we conclude that $[\rho] = [\bar{\rho}]$.
Hence, the character of $\rho$ is real, that is, $\tr \rho(\gamma) \in \RR$ for any $\gamma \in \pi_1(-Y)$.
By Lemma~\ref{lem:real-character-zero-volume}, the volumes of these $10$ representations are zero.

On the other hand, $Y$ is known to be hyperbolic, and hence it has the holonomy representation.
Therefore, the two representations corresponding to complex roots must be the holonomy representation and its complex conjugate.
From Table~\ref{tab:SL2reps}, we have $\tr \rho_1(\lambda) \approx 1.96292$, which is a root of $p_6$.
Since both the holonomy representation and $\rho_1$ arise from the same irreducible polynomial $p_6$, they are Galois conjugate by Lemma~\ref{lem:same_polynomial} below.
\end{proof}

\begin{lemma}
\label{lem:same_polynomial}
Let $j \in \{2,4,6\}$.
If representations $\rho$ and $\rho'$ arise from the polynomial $p_j$, then they are Galois conjugate up to equivalence.
\end{lemma}

\begin{proof}
Let $a$ and $a'$ be the roots of $p_j$ in $\CC$ which provide $\rho$ and $\rho'$, respectively.
Applying Proposition~\ref{prop:main-computation-algebraic} to $\rho$ and $\rho'$, we obtain two number fields, and let $F\subset \CC$ be their composite.
Then the images of $\rho$ and $\rho'$ lie in $\SL(2,F)$.
Since $a$ and $a'$ are both roots of the same irreducible polynomial $p_j$ over $\QQ$, there exists an element $\sigma \in \Gal(\overline{\QQ}/\QQ) = \{\tau \in \mathrm{Aut}(\overline{\QQ}) \mid \tau|_{\QQ}=\mathrm{id}_{\QQ}\}$ satisfying $\sigma(a) = a'$.
It induces a field embedding $\sigma \colon F \hookrightarrow \CC$ and we have a representation $\rho^\sigma$.
By construction, $\tr \rho^\sigma(\lambda) = \tr \rho'(\lambda)$.
Since the $12$ representations are characterized by the trace on $\lambda$, we conclude that $[\rho^\sigma] = [\rho']$.
\end{proof}

\raggedbottom
\bibliographystyle{alpha}
\bibliography{Chern-Simons}

@misc{ADHLP22,
Author = {Paolo Aceto and Aliakbar Daemi and Jennifer Hom and Tye Lidman and JungHwan Park},
Title = {Handle decomposition complexity and representation spaces},
Year = {2024},
howpublished = {arXiv:2210.06607v2},
}

@book{MR03,
  author    = {Maclachlan, Colin and Reid, Alan W.},
  title     = {The Arithmetic of Hyperbolic {3}-Manifolds},
  series    = {Graduate Texts in Mathematics},
  volume    = {219},
  publisher = {Springer},
  address   = {New York},
  year      = {2003},
  doi       = {10.1007/978-1-4757-6720-9},
  isbn      = {978-0-387-98386-8}
}

@misc{Thu80,
  author = {Thurston, William P.},
  title  = {The Geometry and Topology of Three-Manifolds},
  year   = {1980},
  note   = {Princeton University lecture notes, electronic version 1.1},
  url    = {https://library.slmath.org/nonmsri/gt3m/}
}

@article{FS90,
  author  = {Fintushel, Ronald and Stern, Ronald J.},
  title   = {Instanton homology of {S}eifert fibred homology three spheres},
  journal = {Proceedings of the London Mathematical Society},
  series  = {3},
  volume  = {61},
  number  = {1},
  pages   = {109--137},
  year    = {1990},
}

@article{Fur90,
  author  = {Furuta, Mikio},
  title   = {Homology cobordism group of homology 3-spheres},
  journal = {Inventiones Mathematicae},
  volume  = {100},
  number  = {2},
  pages   = {339--355},
  year    = {1990},
}

@article{Fro02,
  author  = {Fr{\o}yshov, Kim A.},
  title   = {Equivariant aspects of {Y}ang--{M}ills {F}loer theory},
  journal = {Topology},
  volume  = {41},
  number  = {3},
  pages   = {525--552},
  year    = {2002},
}

@article{OS03,
  author  = {Ozsv{\'a}th, Peter and Szab{\'o}, Zolt{\'a}n},
  title   = {Absolutely graded {F}loer homologies and intersection forms for four-manifolds with boundary},
  journal = {Advances in Mathematics},
  volume  = {173},
  number  = {2},
  pages   = {179--261},
  year    = {2003},
}

@article{HHL21,
  author  = {Hendricks, Kristen and Hom, Jennifer and Lidman, Tye},
  title   = {Applications of involutive {H}eegaard {F}loer homology},
  journal = {Journal of the Institute of Mathematics of Jussieu},
  volume  = {20},
  number  = {1},
  pages   = {187--224},
  year    = {2021},
  doi     = {10.1017/S147474801900015X},
}

@article{HM17,
  author  = {Hendricks, Kristen and Manolescu, Ciprian},
  title   = {Involutive {H}eegaard {F}loer homology},
  journal = {Duke Mathematical Journal},
  volume  = {166},
  number  = {7},
  pages   = {1211--1299},
  year    = {2017},
  doi     = {10.1215/00127094-3793141},
}

@article{DM19,
  author  = {Dai, Irving and Manolescu, Ciprian},
  title   = {Involutive {H}eegaard {F}loer homology and plumbed three-manifolds},
  journal = {Journal of the Institute of Mathematics of Jussieu},
  volume  = {18},
  number  = {6},
  pages   = {1115--1155},
  year    = {2019},
  doi     = {10.1017/S1474748017000329},
}

@article{BS26,
  author  = {Baldwin, John A. and Sivek, Steven},
  title   = {Framed instanton homology and concordance, {II}},
  journal = {Transactions of the American Mathematical Society},
  volume  = {379},
  number  = {2},
  pages   = {825--865},
  year    = {2026},
  doi     = {10.1090/tran/9320}
}

@article{FS85,
  author  = {Fintushel, Ronald and Stern, Ronald J.},
  title   = {Pseudofree orbifolds},
  journal = {Annals of Mathematics},
  volume  = {122},
  number  = {2},
  pages   = {335--364},
  year    = {1985},
  doi     = {10.2307/1971306}
}

@article {Man03,
    AUTHOR = {Manolescu, Ciprian},
     TITLE = {Seiberg-{W}itten-{F}loer stable homotopy type of
              three-manifolds with {$b_1=0$}},
   JOURNAL = {Geom. Topol.},
  FJOURNAL = {Geometry and Topology},
    VOLUME = {7},
      YEAR = {2003},
     PAGES = {889--932},
      ISSN = {1465-3060,1364-0380},
   MRCLASS = {57R58 (37B30 57R57)},
  MRNUMBER = {2026550},
MRREVIEWER = {Klaus\ Mohnke},
       DOI = {10.2140/gt.2003.7.889},
       URL = {https://doi.org/10.2140/gt.2003.7.889},
}

@article{Man14,
  author  = {Manolescu, Ciprian},
  title   = {On the intersection forms of spin four-manifolds with boundary},
  journal = {Mathematische Annalen},
  volume  = {359},
  number  = {3--4},
  pages   = {695--728},
  year    = {2014},
  doi     = {10.1007/s00208-014-1010-1}
}

@article {Fro10,
    AUTHOR = {Fr{\o}yshov, Kim A.},
     TITLE = {Monopole {F}loer homology for rational homology 3-spheres},
   JOURNAL = {Duke Math. J.},
  FJOURNAL = {Duke Mathematical Journal},
    VOLUME = {155},
      YEAR = {2010},
    NUMBER = {3},
     PAGES = {519--576},
      ISSN = {0012-7094,1547-7398},
   MRCLASS = {57R58 (57R57)},
  MRNUMBER = {2738582},
MRREVIEWER = {Brendan\ E.\ Owens},
       DOI = {10.1215/00127094-2010-060},
       URL = {https://doi.org/10.1215/00127094-2010-060},
}

@article{FF00,
  author  = {Fukumoto, Yoshihiro and Furuta, Mikio},
  title   = {Homology 3-spheres bounding acyclic 4-manifolds},
  journal = {Mathematical Research Letters},
  volume  = {7},
  number  = {6},
  pages   = {757--766},
  year    = {2000},
  doi     = {10.4310/MRL.2000.v7.n6.a8}
}

@incollection{Mat78,
  author = {Matumoto, Takao},
  title = {Triangulation of manifolds},
  booktitle = {Algebraic and Geometric Topology Part 2},
  series = {Proceedings of Symposia in Pure Mathematics},
  volume = {32},
  pages = {3--6},
  publisher = {American Mathematical Society},
  address = {Providence, RI},
  year = {1978}
}

@article{GS80,
  author = {Galewski, David E. and Stern, Ronald J.},
  title = {Classification of simplicial triangulations of topological manifolds},
  journal = {Ann. of Math. (2)},
  volume = {111},
  number = {1},
  pages = {1--34},
  year = {1980},
  doi = {10.2307/1971215}
}

@article {Gro82,
    AUTHOR = {Gromov, Michael},
     TITLE = {Volume and bounded cohomology},
   JOURNAL = {Inst. Hautes \'Etudes Sci. Publ. Math.},
  FJOURNAL = {Institut des Hautes \'Etudes Scientifiques. Publications
              Math\'ematiques},
    VOLUME = {56},
      YEAR = {1982},
     PAGES = {5--99},
      ISSN = {0073-8301,1618-1913},
   MRCLASS = {53C20 (53C21 57R99 58E99)},
  MRNUMBER = {686042},
MRREVIEWER = {Karsten\ Grove},
       URL = {http://www.numdam.org/item?id=PMIHES_1982__56__5_0},
}

@misc{Fro23q2,
  author = {Fr{\o}yshov, Kim A.},
  title = {Mod 2 instanton homology and 4-manifolds with boundary},
  howpublished = {arXiv:2307.03950v2},
  year = {2024}
}

@article{FLin25,
  author  = {Lin, Francesco},
  title   = {Homology cobordism and the geometry of hyperbolic three-manifolds},
  journal = {Advances in Mathematics},
  volume  = {462},
  pages   = {110087},
  year    = {2025},
  doi     = {10.1016/j.aim.2024.110087}
}

@article{DIST25,
  author  = {Daemi, Aliakbar and Imori, Hayato and Sato, Kouki
             and Scaduto, Christopher and Taniguchi, Masaki},
  title   = {Instantons, special cycles and knot concordance},
  journal = {Geometry \& Topology},
  volume  = {29},
  number  = {8},
  pages   = {4189--4298},
  year    = {2025},
  doi     = {10.2140/gt.2025.29.4189}
}

@article{Sto17,
  author  = {Stoffregen, Matthew},
  title   = {Manolescu invariants of connected sums},
  journal = {Proceedings of the London Mathematical Society},
  series  = {3},
  volume  = {115},
  number  = {5},
  pages   = {1072--1117},
  year    = {2017},
  doi     = {10.1112/plms.12060}
}

@article{LS26,
  author  = {Lee, Jaewon and {\c S}avk, O{\u g}uz},
  title   = {On homology spheres of the trivial local equivalence class},
  journal = {Mathematical Proceedings of the Cambridge Philosophical Society},
  year    = {2026},
  doi     = {10.1017/S0305004126102096},
}

@article{HMZ18,
  author  = {Hendricks, Kristen and Manolescu, Ciprian and Zemke, Ian},
  title   = {A connected sum formula for involutive {H}eegaard {F}loer homology},
  journal = {Selecta Mathematica. New Series},
  volume  = {24},
  number  = {2},
  pages   = {1183--1245},
  year    = {2018},
}

@article{HSZ26,
  author  = {Hendricks, Kristen and Stoffregen, Matthew and Zemke, Ian},
  title   = {A note on the involutive invariants of splices},
  journal = {Algebraic \& Geometric Topology},
  volume  = {26},
  number  = {5},
  pages   = {1907--1921},
  year    = {2026},
  doi     = {10.2140/agt.2026.26.1907},
}

@misc{SnapPy,
     author={Culler, Marc and Dunfield, Nathan M. and Goerner,
     Matthias and Weeks, Jeffrey R.},
     title={Snap{P}y, a computer program for studying the geometry and topology of $3$-manifolds},
     howpublished={Available at \url{http://snappy.computop.org} v3.3.2 (28/07/2026)},
}

@article{Sav24,
  author  = {{\c S}avk, O{\u g}uz},
  title   = {A survey of the homology cobordism group},
  journal = {Bulletin of the American Mathematical Society},
  volume  = {61},
  number  = {1},
  pages   = {119--157},
  year    = {2024},
  doi     = {10.1090/bull/1806}
}

@article{Muk20,
  author  = {Mukherjee, Anubhav},
  title   = {A note on embeddings of 3-manifolds in symplectic 4-manifolds},
  journal = {Algebraic \& Geometric Topology},
  volume  = {25},
  number  = {6},
  pages   = {3251--3270},
  year    = {2025}
}

@article {Sto20,
    AUTHOR = {Stoffregen, Matthew},
     TITLE = {Pin(2)-equivariant {S}eiberg-{W}itten {F}loer homology of
              {S}eifert fibrations},
   JOURNAL = {Compos. Math.},
  FJOURNAL = {Compositio Mathematica},
    VOLUME = {156},
      YEAR = {2020},
    NUMBER = {2},
     PAGES = {199--250},
      ISSN = {0010-437X,1570-5846},
   MRCLASS = {57K31 (57R58)},
  MRNUMBER = {4044465},
MRREVIEWER = {Nikolai\ N.\ Saveliev},
       DOI = {10.1112/s0010437x19007620},
       URL = {https://doi.org/10.1112/s0010437x19007620},
}

@article{Mye83,
  author  = {Myers, Robert},
  title   = {Homology cobordisms, link concordances, and hyperbolic 3-manifolds},
  journal = {Transactions of the American Mathematical Society},
  volume  = {278},
  number  = {1},
  pages   = {271--288},
  year    = {1983},
  doi     = {10.2307/1999315}
}

@article{neumann1995rationality,
  title={Rationality problems for {K}-theory and {Chern-Simons} invariants of hyperbolic manifolds},
  author={Neumann, Walter D and Yang, Jun},
  journal={Enseignement Math{\'e}matique},
  volume={41},
  pages={281--296},
  year={1995},
  publisher={SWETS \& ZEITLINGER}
}

@article{Gol82,
  author  = {Goldman, William M.},
  title   = {Characteristic classes and representations of discrete subgroups of {Lie} groups},
  journal = {Bulletin of the American Mathematical Society (New Series)},
  volume  = {6},
  number  = {1},
  pages   = {91--94},
  year    = {1982},
  doi     = {10.1090/S0273-0979-1982-14974-6}
}

@article{Ker69,
  author  = {Kervaire, Michel A.},
  title   = {Smooth Homology Spheres and Their Fundamental Groups},
  journal = {Transactions of the American Mathematical Society},
  volume  = {144},
  pages   = {67--72},
  year    = {1969}
}

@article{Fre82,
  author  = {Freedman, Michael H.},
  title   = {The Topology of Four-Dimensional Manifolds},
  journal = {Journal of Differential Geometry},
  volume  = {17},
  number  = {3},
  pages   = {357--453},
  year    = {1982}
}

@book{Eis95,
  author    = {David Eisenbud},
  title     = {Commutative Algebra with a View Toward Algebraic Geometry},
  series    = {Graduate Texts in Mathematics},
  volume    = {150},
  publisher = {Springer-Verlag},
  address   = {New York},
  year      = {1995},
  doi       = {10.1007/978-1-4612-5350-1}
}

@article{KM63,
  author  = {Kervaire, Michel A. and Milnor, John W.},
  title   = {Groups of Homotopy Spheres: I},
  journal = {Annals of Mathematics},
  volume  = {77},
  number  = {3},
  pages   = {504--537},
  year    = {1963}
}

@phdthesis{Gon70,
  author = {Gonz{\'a}lez-Acu{\~n}a, Francisco Javier},
  title  = {On Homology Spheres},
  school = {Princeton University},
  year   = {1970}
}

@article{Dun99,
  author  = {Dunfield, Nathan M.},
  title   = {Cyclic surgery, degrees of maps of character curves, and volume rigidity for hyperbolic manifolds},
  journal = {Inventiones Mathematicae},
  volume  = {136},
  number  = {3},
  pages   = {623--657},
  year    = {1999},
  doi     = {10.1007/s002220050321}
}

@article{HHSZ,
  author  = {Hendricks, Kristen and Hom, Jennifer and Stoffregen, Matthew
             and Zemke, Ian},
  title   = {On the quotient of the homology cobordism group by {Seifert} spaces},
  journal = {Transactions of the American Mathematical Society, Series B},
  volume  = {9},
  year    = {2022},
  pages   = {757--781}
}

@book{K3,
  editor    = {Baykur, R. {\.I}nan{\c c} and Kirby, Robion C. and Ruberman, Daniel},
  title     = {K3: A New Problem List in Low-Dimensional Topology},
  series    = {Mathematical Surveys and Monographs},
  volume    = {295},
  publisher = {American Mathematical Society},
  year      = {2026},
  isbn      = {978-1-4704-8433-0}
}

@article{GZ07,
  author  = {Goette, Sebastian and Zickert, Christian K.},
  title   = {The Extended {Bloch} Group and the {Cheeger--Chern--Simons} Class},
  journal = {Geometry \& Topology},
  volume  = {11},
  year    = {2007}
}

@article{Rez96,
  author  = {Reznikov, Alexander},
  title   = {Rationality of secondary classes},
  journal = {Journal of Differential Geometry},
  volume  = {43},
  number  = {3},
  pages   = {674--692},
  year    = {1996},
  doi     = {10.4310/jdg/1214458328}
}

@article{Man16,
	author = {Manolescu, Ciprian},
	doi = {10.1090/jams829},
	fjournal = {Journal of the American Mathematical Society},
	issn = {0894-0347,1088-6834},
	journal = {J. Amer. Math. Soc.},
	mrclass = {57R58 (57M27 57Q15)},
	mrnumber = {3402697},
	mrreviewer = {Tirasan\ Khandhawit},
	number = {1},
	pages = {147--176},
	title = {Pin(2)-equivariant {S}eiberg-{W}itten {F}loer homology and the triangulation conjecture},
	url = {https://doi.org/10.1090/jams829},
	volume = {29},
	year = {2016}}

@article{Tan22,
  author  = {Taniguchi, Masaki},
  title   = {Seifert hypersurfaces of 2-knots and {Chern--Simons} functional},
  journal = {Quantum Topology},
  volume  = {13},
  number  = {2},
  pages   = {335--405},
  year    = {2022},
  doi     = {10.4171/QT/165}
}

@article {KirkKlassen90,
    AUTHOR = {Kirk, Paul A. and Klassen, Eric P.},
     TITLE = {Chern-{S}imons invariants of {$3$}-manifolds and
              representation spaces of knot groups},
   JOURNAL = {Math. Ann.},
  FJOURNAL = {Mathematische Annalen},
    VOLUME = {287},
      YEAR = {1990},
    NUMBER = {2},
     PAGES = {343--367},
      ISSN = {0025-5831,1432-1807},
   MRCLASS = {57M99 (53C05)},
  MRNUMBER = {1054574},
MRREVIEWER = {Michael\ Kapovich},
       DOI = {10.1007/BF01446898},
       URL = {https://doi.org/10.1007/BF01446898},
}

@article {KiNo22,
    AUTHOR = {Kitano, Teruaki and Nozaki, Yuta},
     TITLE = {An algebraic property of {R}eidemeister torsion},
   JOURNAL = {Trans. London Math. Soc.},
  FJOURNAL = {Transactions of the London Mathematical Society},
    VOLUME = {9},
      YEAR = {2022},
    NUMBER = {1},
     PAGES = {136--157},
      ISSN = {2052-4986},
   MRCLASS = {57K31 (11R04 13P15 57M05 57Q10)},
  MRNUMBER = {4535660},
MRREVIEWER = {Joan\ Porti},
       DOI = {10.1112/tlm3.12049},
       URL = {https://doi.org/10.1112/tlm3.12049},
}

@book{Don02,
  author    = {Donaldson, Simon K.},
  title     = {Floer Homology Groups in Yang--Mills Theory},
  series    = {Cambridge Tracts in Mathematics},
  volume    = {147},
  publisher = {Cambridge University Press},
  address   = {Cambridge},
  year      = {2002},
  note      = {With the assistance of M. Furuta and D. Kotschick},
  mrnumber  = {1883043}
}

@misc{ADMT23,
  author        = {Alfieri, Antonio and Dai, Irving and Mallick, Abhishek
                   and Taniguchi, Masaki},
  title         = {Involutions and the {C}hern-{S}imons filtration in
                   instanton {F}loer homology},
  year          = {2023},
  howpublished  = {2309.02309v2},
}

@article{CCGLS94,
	author = {Cooper, D. and Culler, M. and Gillet, H. and Long, D. D. and Shalen, P. B.},
	doi = {10.1007/BF01231526},
	fjournal = {Inventiones Mathematicae},
	issn = {0020-9910},
	journal = {Invent. Math.},
	mrclass = {57N10 (57M25)},
	mrnumber = {1288467},
	mrreviewer = {Serge L. Tabachnikov},
	number = {1},
	pages = {47--84},
	title = {Plane curves associated to character varieties of {$3$}-manifolds},
	url = {https://doi.org/10.1007/BF01231526},
	volume = {118},
	year = {1994},
}

@article{Ril84,
	author = {Riley, Robert},
	doi = {10.1093/qmath/35.2.191},
	fjournal = {The Quarterly Journal of Mathematics. Oxford. Second Series},
	issn = {0033-5606},
	journal = {Quart. J. Math. Oxford Ser. (2)},
	mrclass = {20F38 (57M25)},
	mrnumber = {745421},
	mrreviewer = {Norbert Wielenberg},
	number = {138},
	pages = {191--208},
	title = {Nonabelian representations of {$2$}-bridge knot groups},
	url = {https://mathscinet.ams.org/mathscinet-getitem?mr=745421},
	volume = {35},
	year = {1984},
}

@article{DHST23,
	author = {Dai, Irving and Hom, Jennifer and Stoffregen, Matthew and Truong, Linh},
	doi = {10.1215/00127094-2022-0082},
	fjournal = {Duke Mathematical Journal},
	issn = {0012-7094,1547-7398},
	journal = {Duke Math. J.},
	mrclass = {57Q60 (57R58)},
	mrnumber = {4654053},
	mrreviewer = {Greg\ Friedman},
	number = {12},
	pages = {2365--2432},
	title = {An infinite-rank summand of the homology cobordism group},
	url = {https://doi.org/10.1215/00127094-2022-0082},
	volume = {172},
	year = {2023}}

@article{Dai20,
	author = {Daemi, Aliakbar},
	doi = {10.1215/00127094-2020-0017},
	fjournal = {Duke Mathematical Journal},
	issn = {0012-7094,1547-7398},
	journal = {Duke Math. J.},
	mrclass = {57R58 (57R57)},
	mrnumber = {4158669},
	mrreviewer = {Jianfeng\ Lin},
	number = {15},
	pages = {2827--2886},
	title = {Chern-{S}imons functional and the homology cobordism group},
	url = {https://doi.org/10.1215/00127094-2020-0017},
	volume = {169},
	year = {2020}}

@article{Neu04,
	author = {Neumann, Walter D.},
	doi = {10.2140/gt.2004.8.413},
	fjournal = {Geometry and Topology},
	issn = {1465-3060,1364-0380},
	journal = {Geom. Topol.},
	mrclass = {57M27 (11G55 19E99 57R20 57T99)},
	mrnumber = {2033484},
	mrreviewer = {Kerry\ N.\ Jones},
	pages = {413--474},
	title = {Extended {B}loch group and the {C}heeger-{C}hern-{S}imons class},
	url = {https://doi.org/10.2140/gt.2004.8.413},
	volume = {8},
	year = {2004}}

@article{NST24,
	author = {Nozaki, Yuta and Sato, Kouki and Taniguchi, Masaki},
	doi = {10.4171/jems/1371},
	fjournal = {Journal of the European Mathematical Society (JEMS)},
	issn = {1435-9855,1435-9863},
	journal = {J. Eur. Math. Soc. (JEMS)},
	mrclass = {57Q60 (57K31 57R58 58J28 81T13)},
	mrnumber = {4780492},
	mrreviewer = {Matthew\ Stoffregen},
	number = {12},
	pages = {4699--4761},
	title = {Filtered instanton {F}loer homology and the homology cobordism group},
	url = {https://doi.org/10.4171/jems/1371},
	volume = {26},
	year = {2024}}

@article {Som81,
    AUTHOR = {Soma, Teruhiko},
     TITLE = {The {G}romov invariant of links},
   JOURNAL = {Invent. Math.},
  FJOURNAL = {Inventiones Mathematicae},
    VOLUME = {64},
      YEAR = {1981},
    NUMBER = {3},
     PAGES = {445--454},
      ISSN = {0020-9910,1432-1297},
   MRCLASS = {57N10 (51M25 58F09)},
  MRNUMBER = {632984},
MRREVIEWER = {John\ Hempel},
       DOI = {10.1007/BF01389276},
       URL = {https://doi.org/10.1007/BF01389276},
}

@book {Fri17,
    AUTHOR = {Frigerio, Roberto},
     TITLE = {Bounded cohomology of discrete groups},
    SERIES = {Mathematical Surveys and Monographs},
    VOLUME = {227},
 PUBLISHER = {American Mathematical Society, Providence, RI},
      YEAR = {2017},
     PAGES = {xvi+193},
      ISBN = {978-1-4704-4146-3},
   MRCLASS = {57T10 (20J06 37C85 55N10 57M07 57N16)},
  MRNUMBER = {3726870},
MRREVIEWER = {Clara\ L\"oh},
       DOI = {10.1090/surv/227},
       URL = {https://doi.org/10.1090/surv/227},
}

@article {CoLo96,
    AUTHOR = {Cooper, D. and Long, D. D.},
     TITLE = {Remarks on the {$A$}-polynomial of a knot},
   JOURNAL = {J. Knot Theory Ramifications},
  FJOURNAL = {Journal of Knot Theory and its Ramifications},
    VOLUME = {5},
      YEAR = {1996},
    NUMBER = {5},
     PAGES = {609--628},
      ISSN = {0218-2165,1793-6527},
   MRCLASS = {57M25},
  MRNUMBER = {1414090},
MRREVIEWER = {Mark\ E.\ Kidwell},
       DOI = {10.1142/S0218216596000357},
       URL = {https://doi.org/10.1142/S0218216596000357},
}

\end{document}